\documentclass[11pt,a4paper]{article}
\usepackage{amsfonts, amsmath, amssymb,  amsthm, color}
\usepackage{latexsym}
\title{
The artificial compressibility approximation for MHD equations in unbounded domain}
\author{\textbf{Donatella Donatelli}\\
        {\small Dipartimento di Ingegneria e Scienze dell'Informazione e Matematica}\\
       {\small Universit\`a degli Studi dell'Aquila}\\
       {\small 67100 L'Aquila, Italy}\\
        $\scriptstyle\mathtt{\{donatell@univaq.it\}}$\\}
\date{}
\newcommand{\e}{\varepsilon}		       
\newcommand{\R}{\mathbb{R}}

\newcommand{\ue}{u^{\varepsilon}}
\newcommand{\bu}{{\bf u}}
\newcommand{\bue}{{\bu^{\varepsilon}}}
\newcommand{\bB}{{\bf B}}
\newcommand{\bZ}{{\bf Z}}
\newcommand{\bBe}{{\bB^{\varepsilon}}}
\newcommand{\ut}{\tilde{\bu}}
\newcommand{\Bt}{\tilde{\bB}}
\newcommand{\pe}{p^{\varepsilon}}
\newcommand{\phie}{\phi^{\varepsilon}}
\newcommand{\phit}{\tilde{\phi}}
\newcommand{\pt}{\tilde{p}}
\newcommand{\dive}{\mathop{\mathrm {div}}}
\newcommand{\curl}{\mathop{\mathrm {curl}}}

\newtheorem{theorem}{Theorem}[section]   

\newtheorem{lemma}[theorem]{Lemma}
\newtheorem{proposition}[theorem]{Proposition}

\theoremstyle{definition}
\newtheorem{definition}[theorem]{Definition}

\begin{document}

\maketitle
\begin{abstract}
In this paper we analyze a method of   to approximation for  the  weak solutions of the incompressible magnetohydrodynamic equations (MHD) in unbounded domains. In particular we describe an hyperbolic version of the so called artificial compressibility method adapted to the MHD system. By exploiting the wave equation structure of the approximating system we achieve the convergence of the approximating sequences by means of dispersive  estimate of Strichartz type. We prove that the soleinoidal component of the approximating  velocity and magnetic fields is relatively compact and converges strongly to a  weak solution of the MHD equation. 
\medbreak 
\textbf{Key words and phrases:} 
magnetohydrodynamic equations; hyperbolic equations; wave equations.
\medbreak
\textbf{1991 Mathematics Subject Classification.} Primary 35L65; Secondary
35L40, 76R50.
\end{abstract}
\tableofcontents
\newpage
\section{Introduction}

This paper is concerned with the analysis and description of  a type of approximation for the weak solutions of the $3-D$ incompressible magnetohydrodynamic equations (MHD). The magnetohydrodynamic system describes the macroscopic state of an electrically conductor, viscous and incompressible fluid and consists of   the following set of equations
\begin{equation}\label{MHD}
\begin{cases}
\displaystyle{\partial_{t}\bu-\frac{1}{R_{e}}\Delta \bu+(\bu\cdot\nabla)\bu-S(\bB\cdot\nabla)\bB+\nabla \pi+\frac{S}{2}\nabla|\bB|^{2} = 0}\\
 \displaystyle{\partial_{t}\bB-\frac{1}{R_{m}}\Delta \bB+(\bu\cdot\nabla)\bB-(\bB\cdot\nabla)\bu = 0}\\
\dive\bu =0\\
\dive\bB =0
\end{cases}
\end{equation}
endowed with the following initial data
\begin{equation*}
\bu(x,0) = \bu_{0}(x),\quad \bB(x,0) =\bB_{0}(x),
\end{equation*}
where $x\in \R^{3}$, $t\geq 0$. Here $\bu$, $p$ and $\bB$ are a-dimensional quantities corresponding respectively to the velocity of the fluid, its pressure and the magnetic field. The a-dimensional number $R_{e}$ is the Reynolds number, $R_{m}$ is the magnetic Reynold number,  $M$ is the Hartmann number and $S$ is defined as $S=\frac{M^2}{R_{e}R_{m}}$. \\
For simplicity from now on  we set 
\begin{equation*}
R_{e}=R_{m}=S=1,\quad p=\pi+\frac{|B|^2}{2}.
\end{equation*}
 In particular the first equation in \eqref{MHD} is the balance of the momentum and the second is the equation for the induction,  the equation $\dive\bu=0$ is the incompressibility constraint for the fluid and the equation $\dive\bB=0$ derives from Maxwell equations, see \cite{LL}.\\
 There have been extensive mathematical studies on the solutions to MHD equations \eqref{MHD}. In particular, Duvaut and Lions \cite{DL72} constructed a global weak solution and a local strong solution to the initial boundary value problem. Properties of such solutions have been examined by Sermange and Temam in \cite{ST83}. Furthermore, some sufficient conditions for smoothness were presented for the weak solution to the MHD equations in \cite{HX05b}, \cite{HX05a}.
For completeness we recall here the notion of  weak solutions for the system \eqref{MHD}.
\begin{definition}
\label{weaksol}
We say that  a pair $\bu, \bB$ is a  weak solution to the MHD equations \eqref{MHD} if they 
satisfy  \eqref{MHD} in the sense of distributions, namely
\begin{align*}
&\int_{0}^{T}\!\!\int_{\R^{d}}\left(\nabla \bu\cdot\nabla\varphi+\bB_{i}\bB_{j}\partial_{i}\varphi_{j} -\bu_{i}\bu_{j}\partial_{i}\varphi_{j}-u\cdot
\frac{\partial \varphi}{\partial t}\right) dxdt
=\int_{\R^{d}}\bu_{0}\cdot \varphi dx,
\end{align*}
\begin{align*}
&\int_{0}^{T}\!\!\int_{\R^{d}}\left(\nabla \bB\cdot\nabla\varphi -\bu_{i}\bB_{j}\partial_{i}\varphi_{j}+\bB_{i}\bu_{j}\partial_{i}\varphi_{j}-\bu\cdot
\frac{\partial \varphi}{\partial t}\right) dxdt
=\int_{\R^{d}}\bB_{0}\cdot \varphi dx,
\end{align*}

for all $\varphi\in C^{\infty}_{0}(\R^{d}\times[0,T])$, $\dive \varphi =0$ and
\begin{equation*}
\dive \bu=0 \quad \dive\bB=0\qquad \text{in $\mathcal{D}'(\R^{d}\times[0,T])$}.
\end{equation*} 
Moreover the following energy inequality holds
\begin{align*}
\frac{1}{2}\int_{\R^{d}}&(|\bu(x,t)|^{2}+|\bB(x,t)|^{2})dx+\int_{0}^{t}\!\!\int_{\R^{d}}(|\nabla \bu(x,s)|^{2}+|\nabla \bB(x,s)|^{2})dxds\\ \leq
&\frac{1}{2}\int_{\R^{d}}(|\bu_{0}|^{2}+|\bB_{0}|^{2})dx ,\qquad \text{for all $t\geq 0$}.
\end{align*}

\end{definition}

Motivated by many applications there have been also considerable efforts to develop numerical approximation methods for the system \eqref{MHD}. One of the major difficulty regards the development of numerical schemes including in an efficient way the incompressibility constraints. In fact the ``$\dive$''  operator can never be exactly zero, discretization errors accumulate at each iteration and  after a significant amount of error accumulation, the approximating algorithm breaks down.
In the case of the incompressible Navier Stokes equations in a bounded domain, Chorin \cite{Ch68}, \cite{Ch69}, Temam \cite{Tem69a}, \cite{Tem69b} and Oskolkov \cite{Osk},  to overcome the computational  difficulties connected with   the incompressibility constraints introduced what they named {\em ``artificial compressibility approximation''}. They considered a family of perturbed systems, depending on a positive parameter $\e$, which approximate in the limit the Navier Stokes equation and which contain a sort of ``linearized'' compressibility condition, namely the following system
\begin{equation}
\begin{cases}
\displaystyle{\partial_{t}\bue+\nabla \pe=\mu\Delta \bue-\left(\bue\cdot\nabla\right)\bue -\frac{1}{2}(\dive \bue)\bue}\\
\e\partial_{t}\pe+ \dive \bue=0,
\end{cases}
\label{1.1}
\end{equation}
where $\e$ is the {\em artificial compressibility}.
The papers of Temam \cite{Tem69a}, \cite{Tem69b} and his book \cite{Tem01} discuss the convergence of these approximations on bounded domains by exploiting the classical Sobolev compactness embedding and they recover compactness in time by the well known J.L. Lions \cite{L-JL59} method of fractional derivatives. 
The same system was used in \cite{DM06} and in  \cite{DM10}in the case of the Navier Stokes equations in the whole space $\R^{3}$ and in exterior domain and  then modified in a suitable way in \cite{Don08}  for the Navier Stokes Fourier system in $\R^{3}$.
As we can notice the constraint ``$\dive u=0$''  has been replaced by the evolution equation
\begin{equation*}
\partial_{t}\pe=-\frac{1}{\e}\dive\ue,
\end{equation*}
which can be seen as the linearization  around a constant state of the continuity equation in the case of a compressible fluid.
Concerning the first equation of the system \eqref{1.1} we can observe that compared to the equation of the balance of momentum in the incompressible Navier Stokes equations it has the extra term $-1/2(\dive\ue)\ue$ which has been added as a correction  to avoid the paradox of increasing the kinetic energy along the motion. In analogy of what has been done for the incompressible Navier Stokes  we try to adapt this approach to the MHD equations \eqref{MHD}, keeping in mind the {\rm artificial compressibility term} we try to add to our system a similar term for the magnetic field, namely we introduce the following approximating system
\begin{equation}
\begin{cases}
\displaystyle{\partial_{t}\bue+\nabla \pe=\mu\Delta \bue-\left(\bue\cdot\nabla\right)\bue -\frac{1}{2}(\dive \bue)\bue+\curl \bBe \times \bBe}\\
\displaystyle{\partial_{t}\bBe+\nabla \phi^{\varepsilon}=\Delta \bBe+\curl (\bue\times \bBe)}\\
\e\partial_{t}\pe+ \dive \bue=0\\
\e\partial_{t}\phi^{\varepsilon}+\dive \bBe=0
\end{cases}
\label{MHDA}
\end{equation}
where the introduced new variable $\phie$ is a {\em scalar potential function}. In the  incompressible Navier Stokes equations, $\pe$ plays the role of a scalar potential that constrains the velocity to the space of divergence free vector fields, here in the formulation \eqref{MHDA}, the scalar potential, $\phie$, plays the same role but, this time, with the aim to constrain the magnetic field, $\bBe$, to the space of divergence free fields. Hence, an analogous artificial compressibility term, which consists of the time derivative of the scalar potential function, $\phie$, has been added to the solenoidal constraint equation, and similarly, a $\nabla\phie$ term with some other stabilization terms are  included in the induction equation.  A similar approximating system has been introduced in  \cite{YAL}and \cite{YALQD}  for a MHD system without viscosity but the analysis is performed only from a numerical point of view. 
In this paper we analyze  rigorously the convergence of the system \eqref{MHDA} towards the system \eqref{MHD} as $\e\to 0$ in the case of the whole space .
Furthermore we assign to the system \eqref{MHDA} the following  initial conditions
\begin{equation}
\begin{split}
&\bue(x,0)=\bu^{\e}_{0}(x),\  \bBe(x,0)=\bB^{\e}_{0}(x),\\
 &\pe(x,0)=\pe_{0}(x),\ \phie(x,0)=\phie_{0}(x).
\end{split}
\label{3.3}
\end{equation}
The limiting behavior as $\e\downarrow 0$ of the initial data  \eqref{3.3} deserves a little discussion. Indeed \eqref{MHDA} requires the initial conditions \eqref{3.3}
while the MHD equations \eqref{MHD} require only the two initial condition for the velocity $\bu$ and for the magnetic field $\bB$. Hence, our approximation will be consistent if the initial datum on the pressure $\pe$ and of the potential $\phie$ will be eliminated by an ``initial layer'' phenomenon. Since in the limit we have to deal with weak solutions in the sense of the Definition \ref{weaksol} it is reasonable to require the finite energy constraint to be satisfied by the approximating sequences $(\bue,\bBe, \pe, \phie)$. So we can deduce a natural behavior to be imposed on the initial data $(\bu^{\e}_{0}, \bB^{\e}_{0}, \pe_{0}, \phie_{0})$, namely
\begin{equation}
\tag {\bf{ID}}
\left .
\begin{split}
& \bu^{\e}_{0}=\bue(\cdot, 0)\longrightarrow \bu_{0}=\bu(\cdot ,0)\ \text{strongly in}\  L^{2}(\R^{3})\\
& \bB^{\e}_{0}=\bBe(\cdot, 0)\longrightarrow \bB_{0}=\bB(\cdot ,0)\ \text{strongly in}\  L^{2}(\R^{3})\notag\\ 
&\sqrt{\e}\pe_{0}=\sqrt\e \pe(\cdot,0)\longrightarrow 0\  \text{strongly in} \  L^{2}(\R^{3})\notag\\
&\sqrt{\e}\phie_{0}=\sqrt\e \phie(\cdot,0)\longrightarrow 0\  \text{strongly in} \  L^{2}(\R^{3})\notag.
\end{split}
\right \}
\label{ID}
\end{equation}
Let us remark that  the convergence of $\sqrt{\e}\pe_{0}$ and $\sqrt{\e}\phie_{0}$ to $0$ is necessary to avoid  the presence of concentrations of energy in the limit.
Now we can state our main result.

\begin{theorem}
Let $(\bue, \bBe,\pe,\phie)$  be a sequence of weak solution in $\R^{3}$ of the system \eqref{MHDA}, assume that the initial data satisfy (ID). Then 
\begin{itemize}
  \item [\bf{(i)}] There exist $\bu, \bB \in L^{\infty}([0,T];L^{2}(\R^{3}))\cap L^{2}([0,T];\dot H^{1}(\R^{3}))$ such that 
  \begin{equation*}
\bue\rightharpoonup \bu \quad \text{weakly in $L^{2}([0,T];\dot H^{1}(\R^{3}))$}.
\end{equation*}
 \begin{equation*}
\bBe\rightharpoonup \bB \quad \text{weakly in $L^{2}([0,T];\dot H^{1}(\R^{3}))$}.
\end{equation*}
  \item [\bf{(ii)}] The gradient components $Q\bue$, $Q\bBe$ of the vector fields $\bue$, $\bBe$  satisfy respectively
  \begin{equation*}
Q\bue\longrightarrow 0\quad \text{ strongly in $L^{2}([0,T];L^{p}(\R^{3}))$, for any $p\in [4,6)$}.
\end{equation*}
\begin{equation*}
Q\bBe\longrightarrow 0\quad \text{ strongly in $L^{2}([0,T];L^{p}(\R^{3}))$, for any $p\in [4,6)$}.
\end{equation*}
 \item [\bf{(iii)}] The divergence free components $P\bue$,  $P\bBe$ of the vector fields $\bue$, $\bBe$ satisfy
   \begin{equation*}
P\bue\longrightarrow P\bu=\bu\quad \text{strongly  in $L^{2}([0,T];L^{2}_{loc}(\R^{3}))$}.
\end{equation*}
\begin{equation*}
P\bBe\longrightarrow P\bB=\bB\quad \text{strongly  in $L^{2}([0,T];L^{2}_{loc}(\R^{3}))$}.
\end{equation*}
\item [\bf{(iv)}]
 $(\bu=P\bu, \bB=P\bB)$ is a Leray weak solution to the incompressible magnetohydrodynamic system
\begin{equation*}
P(\partial_{t} \bu-\Delta \bu+(\bu\cdot\nabla)\bu-(\bB\cdot\nabla)\bB)=0,
\end{equation*}
\begin{equation*}
\partial_{t} \bB-\Delta \bB+(\bu\cdot\nabla)\bB-(\bB\cdot\nabla)\bu=0,
\end{equation*}
and the following energy inequality holds
\begin{align}
\frac{1}{2}\int_{\R^{3}}(|\bu(x,t)|^{2}&+|\bB(x,t)|^{2})dx+\int_{0}^{T}\!\!\int_{\R^{3}}(|\nabla \bu(x,t)|^{2}+|\nabla \bB(x,t)|^{2})dxdt\notag\\
&\leq 
\frac{1}{2}\int_{\R^{3}}(|\bu(x,0)|^{2}+|\bB(x,0)|^{2})dx.
\label{en}
\end{align}
\end{itemize}
\label{tM}
\end{theorem}
One of the major issues connected with the proof of such kind of result is the presence of acoustic waves for the pressure $\pe$ and for the potential $\phie$ which exhibit fast oscillations when $\e\to 0$. As we can see by considering together the  equations $\eqref{MHDA}_{1}$ and $\eqref{MHDA}_{3}$ or $\eqref{MHDA}_{2}$ and $\eqref{MHDA}_{4}$ as $\e$ goes to 0 at the level of the pressure or of the potential, the acoustic waves propagate with high speed $\frac{1}{\e}$ in the space domain. Because of the fast propagating of the acoustics one expects the velocity $\bue$ and the magnetic field $\bBe$ to converge only weakly. In fact if, for instance,  we project the third equation of \eqref{MHDA} on the space of divergence free vector fields we get  that 
$$\partial_{t}P\bBe=P(-\Delta \bBe+\curl(\bue\times \bBe)),$$
in other words the soleinodal component $P\bBe$ is relatively compact in the time variable. Therefore time oscillations are supported by the gradient part of the field $\bBe$ and are related to the propagation of acoustic waves.
We overcome this trouble by using the dispersion of these waves at infinity yielding the strong convergence of $\bue$  and $\bBe$ in space and time.
In particular the system \eqref{MHDA} will be discussed as a semilinear wave type equation for the pressure function and for the potential function and the dispersive estimates will be carried out by using the $L^{p}$-type estimates due to Strichartz \cite{GV95, KT98, S77}. The connection with the dispersive analysis of the acoustic wave equation has also been considered to study the incompressible limit problem. Similar phenomena appear also in the modeling the Debye screening effect for semi- conductor devices, \cite{DM08}, \cite{DM12}. In particular a similar wave equation structure has been exploited in various way by the paper of P.L.Lions and Masmoudi \cite{L-P.L.M98}, Desjardin, Grenier, Lions, Masmoudi \cite{DGLM}, Desjardin Grenier \cite{DG99}.It is worth to mention here that this type of singular limits from hyperbolic to parabolic systems is not covered and doesn't fits in the general framework of diffusive limits analyzed in \cite{DM04}.\\

This paper is organized as follows. In Section 2 we recall the notations, some mathematical tools needed in the paper and recall same basic definitions. In Section 3 we recover all the uniform bounds that can be deduced by the only use of the energy inequality. Section 4 is devoted to the study of the acoustic equations. In Section 5 we prove the strong convergence of the velocity and magnetic field. Finally in Section 6 we give the proof of the main result. The last Section 7 is devoted to the study of the artificial compressibility approximation  for general unbounded domain, as for example an exterior domain.

\section{Notations}

For convenience of the reader we establish some notations and recall some basic facts that will be useful in the sequel.\\
We will denote by  $\mathcal{D}(\R ^d \times \R_+)$
 the space of test function
$C^{\infty}_{0}(\R^d \times \R_+)$, by $\mathcal{D}'(\R^d \times
\R_+)$ the space of Schwartz distributions and $\langle \cdot, \cdot \rangle$
the duality bracket between $\mathcal{D}'$ and $\mathcal{D}$ and by $\mathcal{M}_{t}X'$ the space $C_{c}^{0}([0,T];X)'$. Moreover
$W^{k,p}(\R^{d})=(I-\Delta)^{-\frac{k}{2}}L^{p}(\R^{d})$ and $H^{k}(\R^{d})=W^{k,2}(\R^{d})$ denote the nonhomogeneous Sobolev spaces for any $1\leq p\leq \infty$ and $k\in \R$. $\dot W^{k,p}(\R^{d})=(I-\Delta)^{-\frac{k}{2}}L^{p}(\R^{d})$ and $\dot H^{k}(\R^{d})=W^{k,2}(\R^{d})$  denote the homogeneous Sobolev spaces. The notations
 $L^{p}_{t}L^{q}_{x}$ and $L^{p}_{t}W^{k,q}_{x}$ will abbreviate respectively  the spaces $L^{p}([0,T];L^{q}(\R^{d}))$, and $L^{p}([0,T];W^{k,q}(\R^{d}))$.\\
We shall denote by $Q$ and $P$ respectively  the Leray's projectors $Q$ on the space of gradients vector fields and $P$ on the space of divergence - free vector fields. Namely
\begin{equation}
Q=\nabla \Delta^{-1}\dive\qquad P=I-Q.
\label{1}
\end{equation} 
Let us remark that   $Q$ and $P$  can be expressed in terms of Riesz multipliers, therefore they are  bounded linear operators on every $W^{k,p}$ $(1< p <\infty)$ space (see \cite{Ste93}).\\
We recall the following vector identities that we will  use later on in the paper. If ${\bf A}$ and ${\bf H}$ are two vector fields then it holds
\begin{equation}
\label{v1}
\nabla(|{\bf A}|^{2})=2({\bf A}\cdot\nabla){\bf A}+2{\bf A}\times\curl {\bf A},
\end{equation}
\begin{equation}
\label{v2}
\curl({{\bf A}\times {\bf H}})={\bf A}\dive {\bf H}-{\bf H}\dive{\bf A}+({\bf H}\cdot\nabla){\bf A}-({\bf A}\cdot\nabla){\bf H},
\end{equation}
\begin{equation}
\dive(({\bf A}\times {\bf H})\times {\bf H})=\curl {\bf H}\times {\bf H}\cdot {\bf A}-\curl({\bf A}\times {\bf H})\cdot {\bf H}.
\label{v3}
\end{equation}

\section{Uniform bounds}
In this section we wish to establish the priori estimates independent on $\e$, that comes directly by the energy associated to the system  \eqref{MHDA}. In fact we define  the energy of the system \eqref{MHDA} as 
\begin{equation}
\label{E}
E(t)=\frac{1}{2}\int_{\R^{3}}\left( |\bue(x,t)|^{2}+ |\bBe(x,t)|^{2}+ | \sqrt{\e}\pe(x,t)|^{2}+ | \sqrt{\e}\phi^{\e}(x,t)|^{2}\right)dx.
\end{equation}
If we multiply the equations of the system \eqref{MHDA} respectively by $\bue$, $\bBe$,  $\pe$ and $\phi^{\e}$, we sum up and integrate by parts in space and time, with the use of the identities \eqref{v1}-\eqref{v3} we are able to prove the following theorem
\begin{theorem}
Let us consider the solution $(\bue, \bBe, \pe, \phie)$ of the Cauchy problem for the system \eqref{MHDA}. Assume that the hypotheses (ID) hold, then one has
\begin{equation}
\label{9}
E(t)+\int_{0}^{t}\!\!\int_{\R^{3}}\left(|\nabla \bue(x,s)|^{2}+|\nabla \bBe(x,s)|^{2}\right)dxds=E(0).
\end{equation}
\label{t1}
\end{theorem}
From the energy equality \eqref{9} we get the following list of estimates
 \begin{equation}
  \sqrt{\e}\pe, \sqrt{\e}\phie \quad  \text{are bounded in} \ L^{\infty}([0,T];L^{2}(\R^{3})),\label{11}
 \end{equation}
  \begin{equation}
 \nabla\bue, \nabla\bBe \quad  \text{is bounded in $L^{2}([0,T]\times\R^{3}).$}  \label{13}
 \end{equation}
 By combining \eqref{11} with Sobolev embeddings theorems we obtain that
  \begin{equation}
 \bue, \bBe \quad  \text{are bounded in $L^{\infty}([0,T];L^{2}(\R^{3}))\cap L^{2}([0,T];L^{6}(\R^{3})),$}  \label{14}
 \end{equation}
Using together \eqref{13} and \eqref{14} we have that 
  \begin{equation}
  \begin{split}
&(\bue \!\cdot\!\nabla)\bue,  \bue\dive\bue, \bue\times \bBe, \curl \bBe\times\bBe \\
 &\text{are bounded in $L^{2}([0,T];L^{1}(\R^{3}))\cap L^{1}([0,T];L^{3/2}(\R^{3})),$}  
\end{split}
\label{17}
\end{equation}
Finally \eqref{11} combined with the equations $\eqref{MHDA}_{3}$ and $\eqref{MHDA}_{4}$ gives rise to  
\begin{equation}
 \e\pe_{t}, \e\phie_{t} \quad  \text{relatively compact in $H^{-1}([0,T]\times \R^{3}).$}  \label{12}
 \end{equation}
 
\section{Acoustic  wave equations and estimates}
As already mentioned in  the Introduction one of the major issue in this type of analysis is the presence of acoustic waves which prevents the strong convergence of the gradient part of the velocity and magnetic field. It turns out that these waves satisfy  inhomogenous wave equations which can be seen as   the superposition of waves with different frequencies scales connected with the viscous tensor of the fluid and  to the convective terms. In order to study this phenomena first of all we rescale the time variable, the velocity, the magnetic field, the potential $\phie$ and the pressure in the following way 
\begin{equation}
\label{18}
\left .
\begin{split}
\tau&=\frac{t}{\sqrt{\e}}\\
\ut(x,\tau)=\bue(x,\sqrt{\e}\tau), &\quad \Bt(x,\tau)=\bBe(x,\sqrt{\e}\tau)\\
 \pt(x,\tau)=\pe(x,\sqrt{\e}\tau), &\quad \phit(x,\tau)=\phie(x,\sqrt{\e}\tau).
 \end{split}
 \right \}
\end{equation}
As a consequence of this scaling the system \eqref{MHDA} becomes
\begin{equation}
\begin{cases}
\displaystyle{\partial_{\tau}\ut+\sqrt{\e}\nabla \pt=\sqrt{\e}\Delta \ut-\sqrt{\e}\left(\ut\cdot\nabla\right)\ut -\frac{\sqrt{\e}}{2}(\dive \ut)\ut}+\sqrt{\e}\curl \Bt\times \Bt\\
\displaystyle{\partial_{\tau}\Bt+\sqrt{\e}\nabla \phit=\sqrt{\e}\Delta \Bt-\sqrt{\e}\curl( \ut\times \Bt)}\\
\sqrt{\e}\partial_{\tau}\pt+ \dive \ut=0\\
\sqrt{\e}\partial_{\tau}\phit+ \dive \Bt=0
\end{cases}
\label{19}
\end{equation}
then, by differentiating with respect to time the equation $\eqref{19}_{3}$ and by using $\eqref{19}_{1}$, we get that $\pt$ satisfies the following wave equation
\begin{equation}
\label{20}
\partial_{\tau\tau}\pt-\Delta\pt= -\Delta \dive\ut+\dive\left(\left(\ut\cdot\nabla\right)\ut +\frac{1}{2}(\dive \ut)\ut +\curl \Bt\times \Bt\right),
\end{equation}
where, by using the bounds \eqref{11}-\eqref{17}, the inhomogenous parts belongs to 
$L^{2}([0,T];H^{-2}(\R^{3}))+ L^{1}([0,T];W^{-1,3/2}(\R^{3}))$.

In the same way by differentiating with respect to time the equation $\eqref{19}_{4}$ and by using $\eqref{19}_{2}$, we get that $\Bt$ satisfies the following wave equation
\begin{equation}
\label{20B}
\partial_{\tau\tau}\phit-\Delta\phit =-\Delta \dive\Bt,
\end{equation}
with the right-hand side  in  $L^{2}([0,T];H^{-2}(\R^{3}))$.
In order to estimate $\pt$ and $\phit$ we will make use of the dispersive estimates of Strichartz type that we briefly recall in the next section.
\subsection{Strichartz estimates}
If  $w$ is a (weak) solution of the following wave equation in the space $[0,T]\times \R^{d}$
\begin{equation*}
\begin{cases}
\left(-\frac{\partial ^{2}}{\partial t}+\Delta\right)w(t,x)=F(t,x)\\
w(0,\cdot)=f,\quad \partial_{t}w(0,\cdot)=g,
\end{cases}
\end{equation*}
for some data $f,g, F$ and time $0<T<\infty$, 
then $w$ satisfies the following Strichartz estimates, (see \cite{GV95}, \cite{KT98})
\begin{equation}
\|w\|_{L^{q}_{t}L^{r}_{x}}+\|\partial_{t}w\|_{L^{q}_{t}W^{-1,r}_{x}}\lesssim \|f\|_{\dot H^{\gamma}_{x}}+\|g\|_{\dot H^{\gamma -1}_{x}}+\|F\|_{L^{\tilde{q}'}_{t}L^{\tilde{r}'}_{x}},
\label{s2}
\end{equation}
where $(q,r)$, $(\tilde{q},\tilde{r})$ are \emph{wave admissible} pairs, namely they satisfy 
\begin{equation*}
\frac{2}{q}\leq (d-1)\left(\frac{1}{2}-\frac{1}{r}\right) \qquad 
\frac{2}{\tilde{q}}\leq (d-1)\left(\frac{1}{2}-\frac{1}{\tilde{r}}\right)
\end{equation*}
and moreover the following   conditions holds
\begin{equation*}
\frac{1}{q}+\frac{d}{r}=\frac{d}{2}-\gamma=\frac{1}{\tilde{q}'}+\frac{d}{\tilde{r}'}-2.
\end{equation*}
In particular, later on we shall use \eqref{s2} in the case of $d=3$, $(\tilde{q}', {\tilde{r}'})=(1, 3/2)$,  then $\gamma=1/2$ and $(q,r)=(4,4)$, namely the following estimate
\begin{equation}
\|w\|_{L^{4}_{t,x}}+\|\partial_{t}w\|_{L^{4}_{t}W^{-1,4}_{x}}\lesssim \|f\|_{\dot H^{1/2}_{x}}+\|g\|_{\dot H^{ -1/2}_{x}}+\|F\|_{L^{1}_{t}L^{3/2}_{x}}.
\label{s3}
\end{equation}
Beside the Strichartz estimate \eqref{s2} or \eqref{s3} in the case of $d=3$ (see \cite{So95}), by using Duhamel's principle  and an earlier  Strichartz estimate  \cite{S77}, we can also  deduce the  estimate
\begin{equation}
\|w\|_{L^{4}_{t,x}}+\|\partial_{t}w\|_{L^{4}_{t}W^{-1,4}_{x}}\lesssim \|f\|_{\dot H^{1/2}_{x}}+\|g\|_{\dot H^{-1/2}_{x}}+\|F\|_{L^{1}_{t}L^{2}_{x}}.
\label{s1}
\end{equation}

\subsection{Estimates for $\phit$}
In order to apply the dispersive estimates of the type \eqref{s2} to the wave equation \eqref{20B} in different components we have to specify the initial data, namely we will analyze the following wave equations
\begin{equation}
\label{21B}
\begin{cases}
     \partial_{\tau\tau}\phit-\Delta\phit=-\Delta \dive\Bt=G \\
     \phit(x,0)=\phi^{\e}_{0}\quad \partial_{\tau} \phit(x,0)=\dive\bB^{\e}_{0},
   \end{cases}
\end{equation}    
We are able to prove the following theorem.
\begin{theorem}
Let us consider the solution $(\bue, \bBe,\pe,\phie)$ of the Cauchy problem for the system \eqref{MHDA}. Assume that the hypotheses (ID) hold. Then we set the following estimate 
\begin{align}
\hspace{-1mm}\e^{3/8}\|\phie\|_{L^{4}_{t} W^{-2,4}_{t}}+\e^{7/8}\|\partial_{t}\phie\|_{L^{4}_{t} W^{-3,4}_{t}}&\lesssim \sqrt{\e}\|\phie_{0}\|_{L^{2}_{x}}+\|\dive {\bf B}^{\e}_{0}\|_{H^{-1}_{x}}\notag\\
&+\sqrt{T}\|\dive \bBe\|_{L^{2}_{t}L^{2}_{x}}.
\label{23B}
\end{align}
\label{t2B}
\end{theorem}
\begin{proof}
Since $\phit$ and is a  solution of the wave equation \eqref{21B}  and $G$ is bounded in $L^{1}_{\tau}H^{-2}_{x}$ we can apply the Strichartz estimate \eqref{s1}, with $(x,\tau)\in \R^{3}\times\left (0,T/\sqrt \e\right)$ and we get
\begin{align}
\label{25B}
\|\Delta^{-1}\phit\|_{L^{4}_{\tau,x}}+\|\partial_{\tau}\Delta^{-1}\phit\|_{L^{4}_{\tau} W^{-1,4}}&\lesssim \|\Delta^{-1}\phit(x,0)\|_{ H^{1/2}_{x}}+
\|\Delta^{-1/2}\partial_{\tau}\phit(x,0)\|_{ H^{-1/2}_{x}}\notag\\&+\|\Delta^{-1}G\|_{L^{1}_{\tau}L^{3/2}_{x},}
\end{align}
namely
\begin{align}
\label{26B}
\|\phit\|_{L^{4}_{\tau} W^{-2,4}_{x}}+\|\partial_{\tau}\phit\|_{L^{4}_{\tau} W^{-3,4}_{x}}&\lesssim \|\phit(x,0)\|_{ H^{-3/2}_{x}}+
\|\partial_{\tau}\phit(x,0)\|_{ H^{-5/2}_{x}}\notag
\\&+\frac{\sqrt{T}}{\e^{1/4}}\|\dive \Bt\|_{L^{2}_{\tau}L^{2}_{x}}
\end{align}
Finally, since
\begin{equation*}
\|\phit\|_{L^{r}((0,T/\sqrt \e );L^{q}(\R^{3}))}=\e^{-1/2r}\|\phie\|_{L^{r}([0,T];L^{q}(\R^{3}))}
\end{equation*}
 we end up with \eqref{23B}.
\end{proof}

\subsection{Estimates for $\pt$}

In order to get the estimates for $\pt$ we split it as $\pt=\pt_{1}+\pt_{2}$ where $\pt_{1}$ and $\pt_{2}$  solve the following wave equations:
\begin{equation}
\label{21}
\begin{cases}
     \partial_{\tau\tau}\pt_{1}-\Delta\pt_{1} =-\Delta \dive\ut=F_{1} \\
     \pt_{1}(x,0)=\partial_{\tau} \pt_{1}(x,0)=0,
   \end{cases}
\end{equation}    
\begin{equation}
\label{22} 
\begin{cases}   
\displaystyle{ \partial_{\tau\tau}\pt_{2}-\Delta\pt_{2} =\dive \left(\left(\ut\cdot\nabla\right)\ut +\frac{1}{2}(\dive \ut)\ut + \curl \Bt\times \Bt\right)=F_{2}}\\
\pt_{2}(x,0)=\pt(x,0)\quad  \partial_{\tau}\pt_{2}(x,0)=\partial_{\tau}\pt(x,0).
 \end{cases}
 \end{equation}
Now by following the same line of arguments as in the previous section or as in \cite{DM06}
we can prove the following theorem.
\begin{theorem}
Let us consider the solution  $(\bue, \bBe,\pe,\phie)$ of the Cauchy problem for the system \eqref{MHDA}. Assume that the hypotheses (ID) hold. Then we set the following estimate 
\begin{align}
\hspace{-1mm}\e^{3/8}\|\pe\|_{L^{4}_{t} W^{-2,4}_{t}}+\e^{7/8}\|\partial_{t}\pe\|_{L^{4}_{t} W^{-3,4}_{t}}&\lesssim \sqrt{\e}\|\pe_{0}\|_{L^{2}_{x}}+\|\dive{\bf u}^{\e}_{0}\|_{H^{-1}_{x}}+\sqrt{T}\|\dive \bue\|_{L^{2}_{t}L^{2}_{x}}\notag\\&+
\|\left(\bue\cdot\nabla\right)\bue +\frac{1}{2}(\dive \bue)\bue+\curl \bBe\times \bBe\|_{L^{1}_{t}L^{3/2}_{x}}.
\label{23}
\end{align}
\label{t2}
\end{theorem}

\section{Strong convergence}
This section is devoted to the strong convergence of the magnetic field $\bBe$ and the velocity field $\bue$. In particular we will show that the gradient part  of the magnetic field $Q\bBe$ and of the velocity $Q\bue$ converges strongly to $0$, while the incompressible component of the magnetic field $P\bBe$ of the velocity field $P\bue$ converges strongly to $P\bB=\bB$  and $P\bu=\bu$ respectively , where $\bB$ is the limit profile as $\e\downarrow 0$ of $\bBe$ and $\bu$ is the limit profile as $\e\downarrow 0$ of $\bue$.
We start this section with a short remark on some strong convergence which are an easy consequence of the estimates \eqref{11}, \eqref{23B} and \eqref{23}, whose proof is left to the reader.
\begin{proposition}
\label{c2}
Let us consider the solution $(\bue, \bBe,\pe,\phie)$ of the Cauchy problem for the system \eqref{MHDA}. Assume that the hypotheses (ID) hold. Then, as $\e\downarrow 0$, one has
\begin{align*}
&\e\pe\longrightarrow 0 &\quad& \text{strongly in $L^{\infty}([0,T];L^{2}(\R^{3}))\cap L^{4}([0,T];W^{-2,4}(\R^{3}))$,}
\\
&\e\phie\longrightarrow 0 &\quad& \text{strongly in $L^{\infty}([0,T];L^{2}(\R^{3}))\cap L^{4}([0,T];W^{-2,4}(\R^{3}))$,}
\\
&\dive \bue \longrightarrow 0 &\quad& \text{strongly in $ W^{-1,\infty}([0,T];L^{2}(\R^{3}))\cap L^{4}([0,T];W^{-3,4}(\R^{3}))$}.
\\
&\dive \bBe \longrightarrow 0 &\quad& \text{strongly in $ W^{-1,\infty}([0,T];L^{2}(\R^{3}))\cap L^{4}([0,T];W^{-3,4}(\R^{3}))$}.
\end{align*}
\end{proposition}
Before going into the analysis of these strong convergences we list here some technical results which will use later on.
\begin{lemma}
Let us consider  a smoothing kernel $\psi\in C^{\infty}_{0}(\R^{d})$, such that $\psi\geq 0$, $\int_{\R^{d}}\psi dx=1$, and define
\begin{equation*}
\psi_{\alpha}(x)=\alpha^{-d}\psi\left(\frac{x}{\alpha}\right).
\end{equation*}
Then  for any $f\in \dot H^{1}(\R^{d})$, one has
\begin{equation}
\label{y1}
\|f-f\ast \psi_{\alpha}\|_{L^{p}(\R^{d})}\leq C_{p}\alpha^{1-\sigma}\|\nabla f\|_{L^{2}(\R^{d})},
\end{equation}
where
\begin{equation*}
p\in [2, \infty)
\quad \text{if $d=2$}, \quad p\in [2, 6] \quad \text{if $d=3$ \ and}\quad \sigma=d\left(\frac{1}{2}-\frac{1}{p}\right).
\end{equation*}
Moreover the following Young type inequality hold
\begin{equation}
\label{y2}
\|f\ast\psi_{\alpha}\|_{L^{p}(\R^{d})}\leq C\alpha^{-s-d\left(\frac{1}{q}-\frac{1}{p}\right)}\|f\|_{W^{-s,q}(\R^{d})},
\end{equation}
for any $p,q\in [1, \infty]$, $q\leq p$,  $s\geq 0$, $\alpha\in(0,1)$.
\label{ly}
\end{lemma}

\begin{proposition}
Let be $\mathcal{F}\subset L^{p}([0,T];B)$,  $1\leq p<\infty$, $B$ a Banach space. $\mathcal{F}$ is relatively compact in  $L^{p}([0,T];B)$ for $1\leq p<\infty$, or in $C([0,T];B)$ for $p=\infty$ if and only if 
\begin{itemize}
\item[{\bf (i)}]
$\displaystyle{\left\{\int_{t_{1}}^{t_{2}}f(t)dt,\ f\in B\right\}}$ is relatively compact in $B$, $0<t_{1}<t_{2}<T$,
\item[{\bf (ii)}]
$\displaystyle{\lim_{h\to 0}\|f(x+h) - f(x)\|_{L^{p}([0, T-h];B)}=0}$ uniformly for any $f \in \mathcal{F}$.
\end{itemize}
\label{S}
\end{proposition}

\subsection{Strong convergence of $Q\bBe$ and $P\bBe$}

\begin{proposition}
\label{p2B}
Let us consider the solution $(\bue, \bBe,\pe,\phie)$ of the Cauchy problem for the system \eqref{MHDA}. Assume that the hypotheses (ID) hold. Then  as $\e\downarrow 0$,
\begin{equation}
Q\bBe \longrightarrow 0 \quad \text{strongly in $ L^{2}([0,T];L^{p}(\R^{3}))$ for any $p\in [4,6)$ }.
\label{33B}
\end{equation}
\end{proposition}
\begin{proof}
In order to prove the convergence \eqref{33B} we split $Q\bBe$ as follows
\begin{equation*}
\|Q\bBe\|_{L^{2}_{t}L^{p}_{x}}\leq \|Q\bBe-Q\bBe\ast \psi_{\alpha}\|_{L^{2}_{t}L^{p}_{x}}+\|Q\bBe\ast \psi_{\alpha}\|_{L^{2}_{t}L^{p}_{x}}=J_{1}+J_{2},
\end{equation*}
where $\psi_{\alpha}$ is the smoothing kernel defined in Lemma \ref{ly}.
Now we estimate separately $J_{1}$ and $J_{2}$. For $J_{1}$ by using \eqref{y1} we get
\begin{equation}
\label{50}
J_{1}\leq \alpha^{1-3\left(\frac{1}{2}-\frac{1}{p}\right)}\left(\int_{0}^{T}\|\nabla Q\bBe(t)\|_{L^{2}_{x}}^{2} dt\right)\leq \alpha^{1-3\left(\frac{1}{2}-\frac{1}{p}\right)}\|\nabla \bBe\|_{L^{2}_{t}L^{2}_{x}}.
\end{equation}
Hence from the identity $Q\bBe=-\e^{1/8}\nabla\Delta^{-1}\e^{7/8}\partial_{t}\phie$ and by the inequality \eqref{y2} we get  $J_{2}$ satisfies the following estimate
\begin{align}
J_{2}&\leq \e^{1/8}\|\nabla\Delta^{-1}\e^{7/8}\partial_{t}\phie\ast\psi\|_{L^{2}_{t}L^{p}_{x}}
\leq \e^{1/8}\alpha^{-2-3\left(\frac{1}{4}-\frac{1}{p}\right)}\|\e^{7/8}\partial_{t}\phie\|_{L^{2}_{t}W^{-3,4}_{x}}\notag\\
&\leq \e^{1/8}\alpha^{-2-3\left(\frac{1}{4}-\frac{1}{p}\right)}T^{1/4}\|\e^{7/8}\partial_{t}\phie\|_{L^{4}_{t}W^{-3,4}_{x}}.
\label{51}
\end{align}
Therefore,  summing up \eqref{50} and \eqref{51} and by using \eqref{13} and \eqref{23}, we conclude  for any $p\in [4,6)$ that
\begin{equation}
\|Q\bBe\|_{L^{2}_{t}L^{p}_{x}}\leq C\alpha^{1-3\left(\frac{1}{2}-\frac{1}{p}\right)}+C_{T}\e^{1/8}\alpha^{-2-3\left(\frac{1}{4}-\frac{1}{p}\right)}.
\end{equation}
Finally we choose $\alpha$ in terms of $\e$ in order that the two terms in the right hand side of the previous inequality have the same order, namely
\begin{equation}
\alpha=\e^{1/18}.
\end{equation}
Therefore we obtain
\begin{equation*}
\displaystyle{\|Q\bBe\|_{L^{2}_{t}L^{p}_{x}}\leq C_{T}\e^{ \frac{6-p}{36p}}\quad \text{for any $p\in [4,6)$.}}
\end{equation*}
\end{proof}
It remains to prove the strong compactness of the incompressible component of the magnetic field.  To achieve this goal we need to look at some time regularity properties of $P\bBe$ and to use the Proposition \ref{S}.

\begin{proposition}
Let us consider the solution $(\bue, \bBe,\pe,\phie)$  of the Cauchy problem for the system \eqref{MHDA}. Assume that the hypotheses (ID) hold. Then  as $\e\downarrow 0$
\label{t3B}
\begin{equation}
P\bBe \longrightarrow P\bB, \qquad \text{strongly in $L^{2}(0,T;L^{2}_{loc}(\R^{3}))$}.
\label{43B}
\end{equation}
\end{proposition}
\begin{proof}
From \eqref{13} an \eqref{14} we know that $P\bBe$ is uniformly bounded  in $L^{2}_{t}\dot{H}^{1}_{x}$. The strong convergence \eqref{43B} follows by applying the Proposition \ref{S}, provided that for  all $h\in(0,1)$, we have
\begin{equation}
\label{34B}
\|P\bBe(t+h)-P\bBe(t)\|_{L^{2}([0,T]\times \R^{3})}\leq C_{T}h^{1/7}.
\end{equation}
The rest of the proof is devoted to show \eqref{34B}. Let us set $\bZ^{\e}=\bBe(t+h)-\bBe(t)$, we have
\begin{align}
\hspace{-0,25 cm}\|P\bBe(t+h)-P\bBe(t)\|^{2}_{L^{2}([0,T]\times \R^{3})}&=\int_{0}^{T}\!\!\int_{\R^{3}}dtdx(P\bZ^{\e})\cdot(P\bZ^{\e}-P\bZ^{\e}\ast \psi_{\alpha})\notag\\&+\int_{0}^{T}\!\!\int_{\R^{3}}dtdx(P\bZ^{\e})\cdot(P\bZ ^{\e}\ast \psi_{\alpha})=I_{1}+I_{2}.
\label{ 35}
\end{align}
By using \eqref{y1} we can estimate $I_{1}$ in the following way
\begin{align}
I_{1}&\leq \|P\bZ^{\e}\|_{L^{\infty}_{t}L^{2}_{x}}\int_{0}^{T}\|P\bZ^{\e}(t)-(P\bZ^{\e}\ast \psi_{\alpha})(t)\|_{L^{2}_{x}}dt\notag\\&\lesssim \alpha T^{1/2}\|\bBe\|_{L^{\infty}_{t}L^{2}_{x}}\|\nabla\bBe\|_{L^{2}_{t,x}}.
\label{36}
\end{align}
Let us reformulate $P\bZ^{\e}$ in integral form by using the equation $\eqref{MHDA}_{2}$, hence
\begin{align}
\hspace{-0.3cm}I_{2}\leq\left|\int_{0}^{T}\!\!\!dt\!\!\int_{\R^{3}}\!\!\!dx \!\!\int_{t}^{t+h}\!\!\!ds(\Delta \bBe-\curl(\bue\times \bBe)) (s,x)\cdot (P\bZ^{\e}\ast \psi_{\alpha})(t,x)\right|.
\label{37}
\end{align}
Then integrating by parts and by using \eqref{y2}, with $p=\infty$ and $q=2$, we deduce
\begin{align}
I_{2}&\leq h\|\nabla\bBe\|^{2}_{L^{2}_{t,x}}+C\alpha^{-5/2}T^{1/2}\|\bBe\|_{L^{\infty}_{t}L^{2}_{x}}\left(\!h\!\int_{t}^{t+h}\!\!\!\|\bue\times \bBe\|^{2}_{L^{1}_{x}}ds\right)^{1/2}\notag\\
&\leq h\|\nabla\bBe\|^{2}_{L^{2}_{t,x}}+C\alpha^{-5/2}T^{1/2}h\|\bBe\|_{L^{\infty}_{t}L^{2}_{x}}\|\bue\times\bBe\|_{L^{2}_{t}L^{1}_{x}}.
\label{38}
\end{align}
Summing up $I_{1}$, $I_{2}$ and by taking into account \eqref{13}, \eqref{14}, \eqref{17},  we have
\begin{equation}
\|P\bBe(t+h)-P\bBe(t)\|^{2}_{L^{2}([0,T]\times \R^{3})}\leq C(\alpha T^{1/2}+h\alpha^{-5/2}T^{1/2}+h),
\label{39}
\end{equation}
and, by choosing $\alpha=h^{2/7}$, we end up with \eqref{34B}.
\end{proof}

\subsection{Strong convergence of $Q\bue$ and $P\bue$}
The strong convergence for $Q\bue$ and $P\bue$ can be performed with the same method as in the previous section. In particular we are able to prove the following propositions.

\begin{proposition}
Let us consider the solution $(\bue, \bBe,\pe,\phie)$ of the Cauchy problem for the system \eqref{MHDA}. Assume that the hypotheses (ID) hold. Then  as $\e\downarrow 0$,
\begin{equation}
Q\bue \longrightarrow 0 \quad \text{strongly in $ L^{2}([0,T];L^{p}(\R^{3}))$ for any $p\in [4,6)$ }.
\label{33}
\end{equation}
\label{p2}
\end{proposition}

\begin{proposition}
Let us consider the solution $(\bue, \bBe,\pe,\phie)$  of the Cauchy problem for the system \eqref{MHDA}. Assume that the hypotheses (ID) hold. Then  as $\e\downarrow 0$
\label{t3}
\begin{equation}
P\bue \longrightarrow P\bu, \qquad \text{strongly in $L^{2}(0,T;L^{2}_{loc}(\R^{3}))$}.
\label{43}
\end{equation}
\end{proposition}

For more details on the proof of the Propositions \ref{p2} and \ref{t3} see also \cite{DM06}, \cite{DM12}.

\section{Proof of the Theorem \ref{tM}}
The weak convergence in {\bf (i)}  is a consequence of the uniform bounds \eqref{13} and \eqref{14}. The strong convergences in {\bf (ii)} and {\bf (iii)} follows from Propositions \ref{t3},  \ref{t3B}, \ref{p2},  \ref{p2B}.
In order to prove {\bf (iv)}  we apply the Leray projector $P$ to the equations $\eqref{MHDA}_{1,2}$,  then, taking into account the identities \eqref{v1}-\eqref{v2} it follows
\begin{equation}
\partial_{t}P\bue-\mu\Delta P\bue+P(\left(\bue\cdot\nabla\right)\bue) +\frac{1}{2}P((\dive \bue)\bue)-P((\bBe\cdot\nabla)\bBe)=0,
\label{pu}
\end{equation}
\begin{equation}
\partial_{t}P\bBe-\Delta P\bBe+\curl (\bue\times \bBe)=0.
\label{pb}
\end{equation}
Now, by testing the equations \eqref{pu} and \eqref{pb} against a test function it is easy to pass into the limit in the linear terms. For what concerns the nonlinear terms it is enough to decompose the vector fields $\bue$ and $\bBe$ in their divergence free part and gradient part end then to pass into the limit by using {\bf (ii)} and {\bf (iii)}. The nonlinear term $\curl (\bue\times \bBe)$ in the equation \eqref{pb} can be handled by combining the previous decomposition with the identity \eqref{v2}. For more details in this step see \cite{DM06}, \cite{DM10}, \cite{DM12}.

Finally we prove the energy inequality \eqref{en}. By using the weak lower semicontinuity of the weak limits, the hypotheses (ID) and by denoting 
by $\chi$ the weak-limit of $\sqrt{\e}\pe $ and by $\Xi$ the weak-limit of $\sqrt{\e}\phie $ ,  we have

$$\frac{1}{2}\int_{\R^{3}}(|\chi|^{2}+|\Xi|^{2})dx$$
$$+\frac{1}{2}\int_{\R^{3}}(|\bu(x,t)|^{2}+|\bB(x,t)|^{2})dx+\int_{0}^{T}\!\!\int_{\R^{3}}(|\nabla \bu(x,t)|^{2}+|\nabla \bB(x,t)|^{2})dxdt$$
$$\leq
\liminf_{\e\to 0}\left(\frac{1}{2}\int_{\R^{3}}(|\sqrt{\e}\pe|^{2}+|\sqrt{\e}\phie|^{2})dx+\frac{1}{2}\int_{\R^{3}}(|\bue(x,t)|^{2}+|\bBe(x,t)|^{2})dx\right)$$
$$+\liminf_{\e\to 0}\int_{0}^{T}\!\!\int_{\R^{3}}(|\nabla \bue(x,t)|^{2}+|\nabla \bBe(x,t)|^{2})dxdt$$
$$=\liminf_{\e\to 0}\frac{1}{2}\int_{\R^{3}}\left(|\bu^{\e}_{0}|^{2}+|\bB^{\e}_{0}|^{2}+|\sqrt{\e}\pe_{0}|^{2}+|\sqrt{\e}\phie_{0}|^{2}\right)dx=\frac{1}{2}\int_{\R^{3}}(|\bu_{0}|^{2}+|\bB_{0}|^{2})dx.
$$

\section{General unbounded domain case}

This last section is devoted the analysis of the {\em artificial compressibility approximation} in the case of a general {\em unbounded domain} as for example the half-space, an exterior domain, an unbounded strip, etc. For what concerns the proof of the energy estimate \eqref{9} and of the strong convergence of the soleinoidal components $P\bue$ and $P\bBe$ of the velocity and magnetic field we can easily adapt the proof in this case. The main difference is in the method used to prove the compactness of the gradient component of $\bue$ and $\bBe$. As already pointed out the main issue is to control the time oscillations due to the presence of acoustic waves. Since we are in an unbounded domain also in this case we have a dispersive phenomena for the acoustic waves but it cannot be handled in the same way as in the previous part. In fact, as is well known, the Strichartz estimates in unbounded domain different than $\R^{d}$ become more delicate and usually require some restrictions on the shape and geometry of the boundary. For instance in the papers \cite{Burq03}, \cite{Met04}, \cite{SmSo00} the domain is required to be star-shaped or non trapping.
Here we carry out an approach which can be applied to any unbounded domain $\Omega$  which satisfy the hypotheses
\begin{itemize}
\item the \emph{point spectrum} of the Neumann Laplacian $\Delta_N$ in
$L^2(\Omega)$ is empty.
\end{itemize}
Although the absence of eigenvalues for the Neumann Laplacian represents, in
general, a delicate and highly unstable problem there are numerous examples of such domains - the whole space $\R^3$, the half-space, exterior domains, unbounded strips, tube-like domains and waveguides.
Similar approach has been used also in \cite{F10a}, \cite{DFN12}.

Here, for simplicity we will perform our analysis for an exterior domain (the proof can be easily adapted to other unbounded domains) end we will show the convergence for the gradient part of the velocity field $\bue$. The convergence for $\bBe$ can be handled in the same way.

We say that $\Omega$ is an exterior domain if it is the complement in $\R^{3}$ of a compact set  (usually called compact obstacle). Since $\Omega$ has a boundary $\partial\Omega$ we have to be careful in defining the Leray projectors. In particular for a vector field ${\bf v}$ we define the Helmholtz decomposition as
\begin{equation}
{\bf v}=P{\bf v}+Q{\bf v},
\end{equation}
where
$$\dive P{\bf v}=0, \ Q{\bf v}=\nabla \Psi, \ \Delta \psi=\dive {\bf v},\ \nabla \Psi\cdot {\bf n}|_{\partial\Omega}= {\bf v}\cdot {\bf n}|_{\partial\Omega},$$
 and ${\bf n}$ is the exterior normal to the boundary $\partial\Omega$.
We consider the system \eqref{MHDA} endowed with the initial conditions \eqref{ID}, we specify the following asymptotic behavior of $\bue$ and $\bBe$
$$\bue\to 0\quad \bBe\to 0\qquad \text{as $|x|\to 0$},$$
We assume that the boundary is impermeable, namely
$$\bue\cdot {\bf n}|_{\partial\Omega}=0\qquad \bBe\cdot {\bf n}|_{\partial\Omega}=0\qquad $$
and the no-slip boundary conditions hold
$$\bue|_{\partial\Omega}=0\qquad  \bBe|_{\partial\Omega}=0.$$
First of all we rewrite the equations $\eqref{MHDA}_{1}$ and $\eqref{MHDA}_{3}$ in terms of the acoustic potential $\Psi^{\e}$, 
\begin{equation}
\begin{cases}
\e\partial_{t}\pe+ \Delta \Psi^{\e}=0\\
\partial_{t}\Psi^{\e}+ \pe=f^{\e},\\
\end{cases}
\label{AE}
\end{equation}
where ${\bf F}^{\e}=\mu\Delta \bue-\left(\bue\cdot\nabla\right)\bue -\frac{1}{2}(\dive \bue)\bue+\curl \bBe \times \bBe$ and $\nabla f^{\e} =Q{\bf F}^{\e}$.
The solutions for \eqref{AE} may be expressed by means of DuhamelÕs formula as
\begin{align}
\label{84}
 \Psi^{\e}(t) &= \frac{1}{2}
\exp \left( {\rm i} \sqrt{ -\Delta_{N}} \frac{t}{\e} \right) \left[ \Psi^{\e}(0) + \frac{{\rm i}}{\sqrt{ -\Delta_{N}}}\pe(0)\right]\\
&+\frac{1}{2}
\exp \left( -{\rm i} \sqrt{ -\Delta_{N}} \frac{t}{\e} \right) \left[ \Psi^{\e}(0) - \frac{{\rm i}}{\sqrt{ -\Delta_{N}}}\pe(0)\right]\notag\\
& +\frac{1}{2}\int_{0}^{t}\left(\exp \left( {\rm i} \sqrt{ -\Delta_{N}} \frac{t-s}{\e} \right)+\exp \left(- {\rm i} \sqrt{ -\Delta_{N}} \frac{t-s}{\e} \right)\right)[f^{\e}(s)]ds,\notag
\end{align}
where $\Delta_{N}$ is the Laplace operator endowed with the homogenous Neumann boundary conditions in the Hilbert space $L^{2}(\Omega)$.  By observing the formula \eqref{84} we see that $ \Psi^{\e}$ oscillates fast in time at the frequency $1/\e$ as soon as $-\Delta_{N}$ possesses a positive eigenvalue. Since we are in an unbounded domain $-\Delta_{N}$  has an empty point spectrum, this allows us to apply the celebrated RAGE theorem, see Cycon et al. \cite{CyFrKiSi}(Theorem 5.8), see also \cite{DFN12}:

\begin{theorem}
Let $H$ be a Hilbert space, ${A}: {\cal D}({A}) \subset H \to H$ a
self-adjoint operator, $C: H \to H$ a compact operator, and $P_c$
the orthogonal projection onto the space of continuity $H_c$ of $A$,
specifically,
$$
H = H_c \oplus {\rm cl}_H \Big\{ {\rm span} \{ w \in H \ | \ w \ \mbox{an eigenvector of} \ A \} \Big\}.
$$
Then
\begin{equation}
\label{rage1}
\left\| \frac{1}{\tau} \int_0^\tau \exp(-{\rm i} tA ) C P_c \exp(
{\rm i} tA ) \ dt \right\|_{{\cal L}(H)} \to 0 \ \mbox{as}\ \tau
\to \infty.
\end{equation}
\label{rage}
\end{theorem}

We apply Theorem \ref{rage} to $H= L^2(\Omega)$, $A = \sqrt{-\Delta_{N}}$, $C = \chi^2 G(-\Delta_{N})$, with $\chi \in C^{\infty}_{0}(\Omega)$, $\chi \geq 0$, $G\in C^{\infty}_{0}(0,\infty)$, $0\leq G\leq 1$. Taking $\tau=1/\e$ in \eqref{rage1} we have
$$\int_{0}^{T}\left\langle  \exp\left(-{\rm i} \frac{t}{\e} \sqrt{-\Delta_{N}}\right)\chi^{2}G(\sqrt{-\Delta_{N}})\exp\left({\rm i} \frac{t}{\e} \sqrt{-\Delta_{N}}\right) X;Y\right\rangle dt$$
$$\leq\omega(\e)\|X\|_{L^{2}(\Omega)}\|Y\|_{L^{2}(\Omega)},$$
where $\omega(\e)\to 0$ as $\e\to 0$. If we take $Y=G(\sqrt{-\Delta_{N}})[X]$ we deduce that
\begin{equation}
\int_{0}^{T}\left\|\chi G(\sqrt{-\Delta_{N}})\exp\left({\rm i} \frac{t}{\e} \sqrt{-\Delta_{N}}\right)[X] \right\|^{2}_{L^{2}(\Omega)}dt\leq \omega(\e)\|X\|^{2}_{L^{2}(\Omega)}
\label{rage2}
\end{equation}
for any $X\in L^{2}(\Omega)$, for more details see \cite{DFN12}.
The relation \eqref{rage2} implies the local decay to zero of $\Psi^{\e}$, namely
$$\|\nabla\Psi^{\e}\|_{L^{2}(K)}\to 0\quad \text{for any compact set $K\subset (0,T)\times\Omega$},$$
hence the strong compactness in space and time for the gradient part of $\bue$.

\bibliographystyle{amsplain}

\end{document}